\documentclass{amsart}
\usepackage{latexsym,amsmath,amssymb,amscd}

\newtheorem{theorem}{Theorem}

\theoremstyle{definition}

\theoremstyle{remark}

\numberwithin{equation}{section}

\def\als{\alpha_s}
\def\fis{\phi_s}
\def\tal{M\sp\alpha}
\def\taf{N\sp\phi}
\def\cstaf{C^*(N\sp\phi)}
\def\lims{\lim\sb{s\in S}}

\def\gm{\gamma}

\def\pals{\{\alpha_s\}\sb{s\in S}}
\def\pfis{\{\phi_s\}\sb{s\in S}}
\def\xs{x^*}

\begin{document}

\title[Lifting fixed points]
{Lifting fixed points\\
 of completely positive semigroups}
\author{Bebe Prunaru}

\address{Institute of Mathematics ``Simion Stoilow'' 
of the Romanian
Academy\\
P.O. Box 1-764
 RO-014700 Bucharest 
Romania}
\email{bebe.prunaru@gmail.com}

\begin{abstract}

Let $\{\phi_s\}_{s\in S}$ be a commutative semigroup
of completely positive, contractive, 
and weak*-continuous
 linear maps acting on a von Neumann algebra 
$N$. Assume there exists a semigroup $\{\alpha_s\}_{s\in S}$ 
of  weak*-continuous  $*$-endomorphisms of some larger von Neumann algebra $M\supset N$
and a projection $p\in M$  with   $N=pMp$
such that 
$\alpha_s(1-p)\le 1-p$
for every $s\in S$ and 
$\phi_s(y)=p\alpha_s(y)p$ for all $y\in N$.
If $\inf_{s\in S}\alpha_s(1-p)=0$ then  we show
that the map $E:M\to N$ defined by $E(x)=pxp$ for $x\in M$
induces a complete isometry between the fixed point spaces of 
$\{\alpha_s\}_{s\in S}$ and $\{\phi_s\}_{s\in S}$.

\end{abstract}


\maketitle

Let  $(S,+,0)$   be  a  commutative semigroup with unit $0.$
Consider the  partial preorder  on $S$
induced by the semigroup structure as follows. 
If $s,t\in S$ then 
$s\le t$ if and only if  there exists $r\in S$ such that $s+r=t$.
If $X$  is a  Hausdorff 
topological space and   $f:S\to X$
is a function,  ,
 then
$\lims f(s)$ denotes its limit along the directed set $(S,\le)$,
 whenever this limit exists.

Let $M$ be a von Neumann algebra.
Let $CP(M)$ denote the semigroup  of  all 
completely positive, contractive  and weak*-continuous 
 linear maps $\beta:M\to M$. 
Let also    $End(M)$ be the semigroup of all weak*-continuous 
*-endomorphisms of $M$.  
A family $\{\beta_s\}\sb{s\in S}\subset CP(M)$ is called a semigroup 
if the map
$s\mapsto\beta_s$
is a unital  homomorphism of semigroups  from $S$
 into $CP(M)$.

Suppose now that  $\pals\subset End(M)$   is a semigroup.
Let $p$ be an  orthogonal  projection in $M$
 such that 
$$\als(1-p)\le 1-p \qquad\forall s\in S.$$
Then one can define, for every $s\in S$,  
a  completely positive  mapping on the von Neumann algebra $N=pMp$ 
as follows:
$$\fis(x)=p\als(x)p \qquad\forall x\in N.$$
It is clear that $\pfis\subset CP(N)$.
A short calculation shows that  
$$\fis(pxp)=p\als(x)p \qquad\forall x\in M,$$
and using this, one can show that 
$\pfis$ is a semigroup. 
According to the terminology used in  Chapter 8 of \cite{Arv2003},
where this construction is given for one-parameter semigroups,
$\pals$ is a dilation of $\pfis$ and   
$p$ is a co-invariant projection for $\pals$. 

We shall prove the following result, 
which shows that, under a suitable minimality condition,
the fixed point spaces of $\pals$ and $\pfis$ are completely isometric.
We point out that the minimality condition is always satisfied by the minimal 
E-dilation of a CP-semigroup as constructed in \cite{Arv2003}.

\begin{theorem} \label{main}

Let $M\subset B(H)$ be a von Neumann algebra 
on some Hilbert space $H$.
Let $(S,+,0)$   be  a  commutative semigroup with unit 
and let $\{\alpha_s\}\sb{s\in S}\subset End(M)$ be a semigroup of 
 weak*-continuous    *-endomorphisms of $M$.
 Let  $p\in M$ be a  projection
such that 
$$ \als(1-p)\le 1-p  \qquad\forall s\in S  $$
and
$$\inf_{t\in S}\alpha_t(1-p)=0.$$

Let $\{\phi_s\}\sb{s\in S}\subset CP(N)$
  be the  compression  of $\{\alpha_s\}\sb{s\in S}$ to $N=pMp$
  defined  by 
$$\fis(x)=p\als(x)p \qquad\forall x\in N.$$
Let 
$$M\sp\alpha=\{x\in M : \alpha_t(x)=x, \forall t\in S\}$$
and
$$N\sp\phi=\{x\in N : \phi_t(x)=x, \forall t\in S\}$$
and let $\cstaf$ be the $C^*$-subalgebra of $N$
generated by $\taf$.
Let 
$E:M\to M$ be defined by
$$E(x)=pxp   \qquad\forall x\in M.$$

Then the following hold true:

\begin{itemize}

\item[(1)] 
For each $y\in C^*(N^\phi)$ 
there   exists the limit (in the strong operator topology)
$$\pi(y)=so-\lims\alpha_s(y)$$
 and  the map
 $y\mapsto\pi(y)$
is a *-homomorphism from $\cstaf$ onto $M\sp\alpha$
 such that 
$(\pi\circ E)(x)=x$ for all  $x\in M\sp\alpha.$

\item[(2)]    
$E$ induces a   complete isometry between  $M\sp\alpha$ 
and $N\sp\phi$.

\item[(3)] 
For  each $y\in C^*(N^\phi)$ 
 there exists the limit
$$\Phi(y)=so-\lims \phi_s(y)$$
 and  the map
 $y\mapsto\Phi(y)$ 
is completely positive, idempotent,   
$Ran(\Phi)=N^\phi$, 
and  
$E\circ\pi=\Phi$ on $\cstaf$.

\end{itemize}

\end{theorem}

\begin{proof}

First,we show that $E(M\sp\alpha)=N\sp\phi.$
It is clear that $E(M\sp\alpha)\subset N\sp\phi.$
Let $\mu$ be an invariant  mean on $S$.
This means that $\mu$ is a state on the von Neumann algebra 
$\ell\sp{\infty}(S)$ of all complex-valued bounded functions on $S$
that remains invariant under translations.
It is well known \cite{Dix1950} that any commutative semigroup 
is amenable, in the sense that 
it admits  invariant  means.

Let $y\in N\sp\phi$.  
For each $\gm$ in the predual $M_*$ of $M$, 
let $f\sb\gamma\in\ell\sp{\infty}(S)$ be defined by
$$ f\sb\gamma(s)=(\als(y),\gm)  \qquad s\in S. $$
Then there exists $z\in M$   such that 
$$(z,\gamma)=\mu(f\sb\gamma)   \qquad \forall\gm\in M_*.$$
Since $\als$ are weak * continuous
and $\pals$ is a semigroup,
it follows that $z\in\tal$.
Moreover $pzp=y$ and this shows that  $E(M^\alpha)=N^\phi$.

In order to go further,
we need to use the minimality assumption
on $\pals$. 
Suppose now that $w\in M^\alpha$. 
Since $\inf_{t\in S}\alpha_t(1-p)=0$
we see that  
  $$so-\lims\als(pwp)=w.$$
Since $E(\tal)=\taf$
it follows that the limit 
$$\pi(y)=so-\lims\alpha_s(y)$$
exists for every $y\in\taf$
and that $\pi\circ E=id$ on $\tal$.
In particular $E$ is completely  isometric on $M\sp\alpha$.
All the other assertions are straightforward
  consequences of what we have already proved.

\end{proof}

This result and its proof provide, in particular, 
an alternate and simplified approach to the lifting
theorem for fixed points of completely positive maps
from \cite{Pr2009}.
In the case when $S$ is either a commutative, countable and cancellative 
semigroup or $S=\mathbb{R}_+^d$ for some $d\ge1$, 
and $\pals$ are unit preserving,
part 2 of Theorem \ref{main}
follows directly from Proposition 4.4 together with Theorem 4.5 from \cite{Cou2007}.
In the case when $\pfis$ is the semigroup induced by the unilateral shift 
on the Hardy space $H^2$, the existence of the limit in part 3 
of Theorem \ref{main} is proved in \cite{BarHal1982}.

We close with the following result which shows that part 3
of the previous theorem holds true even without assuming the existence of a dilation.

\begin{theorem}

Let $N\subset B(H)$ be a von Neumann algebra 
on some Hilbert space $H$.
  Let $(S,+,0)$   be  a  commutative semigroup with unit. 
Let $\{\phi_s\}\sb{s\in S}$
  be a semigroup of completely positive,
  contractive and weak*-continuous linear maps on $N$.
Let 
$$N\sp\phi=\{x\in N : \phi_t(x)=x, \forall t\in S\}$$
and let $\cstaf$ be the $C^*$-subalgebra of $N$
generated by $\taf$.
Then for  each $y\in C^*(N^\phi)$ 
 there exists the strong-operator limit
$$\Phi(y)=so-\lims \phi_s(y)$$
 and  the map
 $y\mapsto\Phi(y)$ 
is completely positive, contractive, idempotent,   
and moreover $Ran(\Phi)=N^\phi$.

\end{theorem}

\begin{proof}

Let $\mu$ be an invariant  mean on $S$.
Let $y\in N$.  
For each $\gm$ in the predual $N_*$ of $N$,
let $f\sb\gamma\in\ell\sp{\infty}(S)$ be defined by
$$ f\sb\gamma(s)=(\fis(y),\gm)  \qquad s\in S. $$
Then there exists $z\in N$   such that 
$$(z,\gamma)=\mu(f\sb\gamma)   \qquad \forall\gm\in N_*.$$
Since $\pfis$ are weak* continuous
and $\pfis$ is a semigroup,
it follows that $z\in\taf$.
Let us denote $z=\rho(y)$.
The mapping
$\rho:N\to N$
is completely positive, contractive, idempotent, and 
$Ran(\rho)=\taf$.
Moreover 
$$\fis\circ\rho=\rho\circ\fis=\rho.$$
Let $\Phi:\cstaf\to\taf$
be the restriction of $\rho$ to $\cstaf$.
A well known result from \cite{CE} shows that 
$$ \Phi(\Phi(x)y)=\Phi(\Phi(x)\Phi(y)) \qquad \forall x,y\in\cstaf. $$
This easily implies that $ker(\Phi)$ is the   
closed left ideal in $\cstaf$ generated by all the operators of the form 
$xy-\Phi(xy)$ 
with $x,y\in\taf.$
Moreover  by polarization we see that 
$ker\Phi$
 is the closed ideal of $\cstaf$  generated by all the operators 
 of the form
$\xs x-\Phi(\xs x)$
with $x\in\taf.$

Let $x\in\taf.$ 
Since $\pfis$ are completely positive, they satisfy the Kadison-Schwarz inequality,therefore
$$\fis(\xs x)-\xs x\ge0$$ for all $s\in S$
hence the net 
$\{\fis(\xs x)\}_{s\in S}$
is monotone increasing.
It follows from the way $\Phi$ is constructed that 
$$\Phi(\xs x)=so-\lims\fis(\xs x).$$
Let $y=\Phi(\xs x)-\xs x$.
It follows that 
$$so-\lims\fis(y)=0.$$
Let $a\in\cstaf$ and let $h\in H.$
Then
$$\|\fis(ay^{1/2})h\|^2\le(\fis(y^{1/2}a^* ay^{1/2})h,h)
\le\|a\|^2 (\fis(y)h,h).$$
This shows that 
$\lims\|\fis(z)h\|=0$ for every $z\in\ker\Phi$
therefore
$$\Phi(w)=so-\lims\fis(w)$$ 
for every $w\in\cstaf$.
This completes the proof.

\end{proof}


\begin{thebibliography}{10000000}



\bibitem{Arv2003} W. Arveson, \textit{Noncommutative dynamics and 
E-semigroups}, Springer Monographs in Mathematics, Springer-Verlag, New York, 2003. 


\bibitem{BarHal1982} J. Barria and P. R. Halmos, 
    \textit{Asymptotic Toeplitz operators}. 
Trans. Amer. Math. Soc. \textbf{273} (1982), 621-630. 

\bibitem{CE}
M.D.~Choi, E.G.~Effros,
\textit{Injectivity and operator spaces}, 
J. Functional Analysis , \textbf{24}  (1977), no.~2, 156--209.


\bibitem{Cou2007} D. Courtney, \textit{Asymptotic lifts of UCP semigroups},
Ph.D. Dissertation, University of California, Berkeley, 2007. 


\bibitem{Dix1950} J. Dixmier, \textit{Les moyennes invariantes dans les semi-groupes et leurs applications}, Acta Sci. Math. (Szeged) \textbf{12} (1950), 585-590. 




\bibitem{Pr2009} B. Prunaru, \textit{Lifting fixed points of completely positive 
mappings}, J. Math. Anal. Appl. \textbf{350} (2009), 333-339.




\end{thebibliography}
 \end{document}